\documentclass[reqno,11pt]{amsproc}

\usepackage{xcolor}
\definecolor{myurlcolor}{rgb}{0,0,0.4}
\definecolor{mycitecolor}{rgb}{0,0.5,0}
\definecolor{myrefcolor}{rgb}{0.5,0,0}
\usepackage[breaklinks=true]{hyperref}
\hypersetup{colorlinks,
linkcolor=myrefcolor,
citecolor=mycitecolor,
urlcolor=myurlcolor}

\usepackage{amssymb}
\usepackage{amsmath}
\usepackage{amsfonts}
\usepackage{amsthm}
\usepackage{microtype}
\usepackage{xcolor}
\usepackage{geometry}
\usepackage{mathrsfs}
\usepackage{mathtools}
\usepackage{thmtools}
\usepackage[capitalize]{cleveref}

\usepackage[backend=biber,style=ieee,sorting=none,backref=true,citestyle=numeric-comp]{biblatex}
\usepackage{tikz,tikz-cd}
\usetikzlibrary{arrows,cd,decorations.markings,arrows.meta}

\tikzset{->-/.style={decoration={markings,mark=at position .5 with {\arrow{>}}},postaction={decorate}}}
\tikzset{-<-/.style={decoration={markings,mark=at position .5 with {\arrow{<}}},postaction={decorate}}}
\tikzstyle{box}=[fill=white, draw=black, shape=rectangle, inner sep=14pt]
\tikzstyle{forward arrow}=[->-]
\tikzstyle{backward arrow}=[-<-]

\newcommand{\beq}{\begin{equation}}
\newcommand{\eeq}{\end{equation}}
\newcommand{\N}{\mathbb{N}}

\newcommand{\R}{\mathbb{R}}

\newcommand{\eps}{\varepsilon}

\renewcommand{\P}[2][]{\mathbb{P}_{#1}\left[#2\right]}
\newcommand{\E}[2][]{\mathbb{E}_{#1}\left[#2\right]}

\newtheorem{dummy}{Dummy}[section]
\newtheorem{theorem}[dummy]{Theorem}
\newtheorem{lemma}[dummy]{Lemma}
\newtheorem{prop}[dummy]{Proposition}

\newtheorem{nota}[dummy]{Notation}
\newtheorem{prob}[dummy]{Problem}
\newtheorem{game}[dummy]{Game}
\newtheorem{conj}[dummy]{Conjecture}
\theoremstyle{remark}
\newtheorem{ex}[dummy]{Example}
\newtheorem{remark}[dummy]{Remark}
\newtheorem{note}[dummy]{Note}
\numberwithin{equation}{section}
\theoremstyle{definition}

\Crefname{equation}{}{}		
\Crefname{theorem}{Theorem}{Theorems}
\Crefname{lemma}{Lemma}{Lemmas}
\Crefname{prop}{Proposition}{Propositions}
\Crefname{cor}{Corollary}{Corollaries}
\Crefname{defn}{Definition}{Definitions}
\Crefname{nota}{Notation}{Notations}
\Crefname{prob}{Problem}{Problems}
\Crefname{game}{Game}{Games}
\Crefname{conj}{Conjecture}{Conjectures}
\Crefname{ex}{Example}{Examples}
\Crefname{remark}{Remark}{Remarks}
\Crefname{note}{Note}{Notes}

\allowdisplaybreaks

\let\originalleft\left
\let\originalright\right
\renewcommand{\left}{\mathopen{}\mathclose\bgroup\originalleft}
\renewcommand{\right}{\aftergroup\egroup\originalright}

\usepackage{enumitem}
\setlist[enumerate]{label=(\roman*),itemsep=5pt,topsep=8pt}
\setlist[itemize]{label=$\triangleright$,itemsep=5pt,topsep=6pt}
\Crefformat{enumi}{#2#1#3}

\newcommand{\newterm}[1]{\textbf{#1}}

\usepackage{hyphenat}

\usepackage{todonotes}

\begin{document}



\title[Maximizing the probability of a linear inequality]{On maximizing the probability of a linear inequality between iid random variables}

\author{Pierre C Bellec}

\address{Department of Statistics, Rutgers University, USA}
\email{pierre.bellec@rutgers.edu}

\author{Tobias Fritz}

\address{Department of Mathematics, University of Innsbruck, Austria}
\email{tobias.fritz@uibk.ac.at}



\thanks{\textit{Acknowledgements.} We thank Stephan Eckstein, Tom Hutchcroft, Daniel Lacker, Lyuben Lichev, Gonzalo Muñoz, Will Sawin, Omer Tamuz and Vincent Yu for generously sharing their ideas, pointers to the literature, and comments on a draft.}

\begin{abstract}
A casino offers the following game. There are three cups each containing a die.	You are being told that the dice in the cups are all the same, but possibly nonstandard.
After you pay \$1 to play, the croupier shakes all three cups and lets you choose one of them. You win \$2 if the die in your cup displays at least the average of the other two, and you lose otherwise. Is this game in your favor? If not, how should the casino design the dice to maximize their profit?

This problem is a special case of the following more general question: given a measurable space $X$ and a bounded measurable function $f : X^n \to \R$, how large can the expectation of $f$ under probability measures of the form $\mu^{\otimes n}$ be? We develop a general method to answer this kind of question. As an example application that is harder than the casino problem, we show that the maximal probability of the event $X_1 + X_2 + X_3 < 2 X_4$ for nonnegative iid random variables lies between $0.400695$ and $0.417$, where the upper bound is obtained by mixed integer linear programming. We conjecture the lower bound to be the exact value.
\end{abstract}

\newgeometry{top=2cm}
\maketitle
\tableofcontents

This paper studies the problem of maximizing the probability of a strict linear inequality
between iid random variables $X_1,X_2,\dots$, where the maximization is over the common distribution $\mu$.
A curious reader seeking new problems might consider the exercise
of proving $\sup_{\mu} \P{2 X_1< X_2 + X_3}=\frac23$,
or the open problem of determining
$\sup_{\mu} \P{X_1 + X_2 + X_3< 2 X_4}$, for which we will provide
bounds in \Cref{example} below.
Before that, let us consider the exercise problem by first narrating it in terms of betting games in a casino, which will lead us to its solution.

\section{The Beat the Average game}
\label{beat_the_average}

If we throw two identical dice $A$ and $B$, then they are equally likely to beat each other:
\begin{equation}
	\label{eq:two_dice_beat}
	\P{A \ge B} = \P{B \ge A}.
\end{equation}
This is obvious not just for standard dice, where the outcomes are $1, \dots, 6$ with equal probability, but also for any other kind of ``non-standard'' dice as long as they are identical.
Formally,~\eqref{eq:two_dice_beat} can be seen by noting that the joint distribution is invariant under swapping the two dice.
In addition, since at least one of the two events $A \ge B$ and $B \ge A$ always happens, it follows that their probability is $\ge \frac{1}{2}$.
In fact, it is a nice exercise in elementary probability to show that
\[
	\P{A \ge B} = \P{B \ge A} = \frac{1 + \P{A = B}}{2}.
\]
Now what happens when we have three identical dice, and we ask whether die $A$ beats the \emph{average} of dice $B$ and $C$? That is, what can we say about the probability
\[
	\P{A \ge \tfrac{B + C}{2}} \: ?
\]
To start, let us consider standard dice first.
In this case, there is another nice symmetry: the distribution of outcomes for a single die is symmetric around its expectation value $3 \frac{1}{2}$.
This implies that we have similar properties as for~\eqref{eq:two_dice_beat}, namely
\begin{equation}
	\label{eq:three_dice_beat}
	\P{A \ge \tfrac{B + C}{2}} = \P{A \le \tfrac{B + C}{2}} > \frac{1}{2}.
\end{equation}
The probability is strictly bounded by $\frac{1}{2}$ because both events occur if $A = \tfrac{B + C}{2}$, and this happens with nonzero probability.

What can we say if the dice are possibly non-standard?
Let's think about this in more entertaining terms:

\begin{game}[Beat the Average, dice version]
	\label{game:beat_the_average}
	There are three cups containing three identical dice.
	After you pay \$1 to play, the croupier shakes all three cups and lets you choose one of them.
	You win \$2 if the die in your cup displays at least the average of the other two, and you lose otherwise.
\end{game}

A first observation is that because of the symmetry, the choice of cup does not matter: we may as well always choose the cup on the left in every round.
The purpose of letting you choose is merely to give you the illusion of control.
Mathematically, deciding whether the game is in your favour reduces to analyzing the first probability in~\eqref{eq:three_dice_beat}.
As we noted above, this probability is $> \frac{1}{2}$ for standard dice.
Thus if the dice are standard dice, then you should play!

But of course, you are smart enough to know that the casino would not offer this game if it was in your favour.
So what tricks could they have up their sleeve to rip you off?
Short of more dubious practices, the only way to turn the tables in their favour is to choose suitable dice.
But which dice should they use in order to maximize the probability of the
the event $A < \tfrac{B + C}{2}$ that you lose, and how large can they make this probability?
Now we have a mathematical problem on our hands:

\begin{prob}[Casino manager problem]
	\label{prob:optimizing_dice}
	What is the largest probability of $A < \tfrac{B + C}{2}$, given that $A, B, C$ are $\N$-valued iid random variables?
	And which distribution achieves this maximum, if it is achievable?
\end{prob}

Here, we have formalized a die as a distribution over natural number outcomes, which amounts to allowing dice with infinitely many sides.
But as we will see, the maximal probability can be approached by a sequence of dice with an increasing finite number of sides.

As mathematicians, we may want to start thinking about the most degenerate case first: what happens for \emph{one-sided} dice?
Or equivalently, with dice which display the same number on all sides?
In this case, our random variables are deterministic, and the iid assumption implies that they must be equal.
Thus the probability in question is $0$: the player always wins!
Using one-sided dice, or dice displaying the same number on all sides, is the worst thing to do for the casino.

For two-sided dice, we can take the outcomes to be $0$ and $1$ without loss of generality.
Let's say that throwing such a die gives $1$ with probability $p$ and $0$ with probability $1 - p$.
Then the casino wins the game if $A = 0$ and ($B = 1$ or $C = 1$), resulting in
\begin{equation}
	\P{A < \tfrac{B + C}{2}} = (1 - p) (1 - (1 - p)^2) = p^3 - 3 p^2 + 2 p.
        \label{three_by_eight}
\end{equation}
This function with domain $[0,1]$ has its maximum at $p = 1 - \frac{1}{\sqrt{3}}$, where it evaluates to
\[
	\P{A < \tfrac{B + C}{2}} = \frac{2}{3\sqrt{3}} \approx 0.385.
\]
So with two-sided dice, the casino can do better than with one-sided ones, but will still end up losing on average.

Determining the optimal three-sided dice is already a more involved calculation and we will not perform it here.
But it is instructive to calculate the winning probability for a three-sided dice with values $1, 2, 4$ having probability $\frac{1}{3}$ each.
There are $3^3 = 27$ possible outcomes in this case, and the winning ones for the casino are precisely the following combinations:
\begin{center}
	\begin{tabular}{c|ccccccccccccc}
		$A$ & 1 & 1 & 1 & 1 & 1 & 1 & 1 & 1 & 2 & 2 & 2 & 2 & 2 \\\hline
		$B$ & 1 & 1 & 2 & 2 & 2 & 4 & 4 & 4 & 1 & 2 & 4 & 4 & 4 \\\hline
		$C$ & 2 & 4 & 1 & 2 & 4 & 1 & 2 & 4 & 4 & 4 & 1 & 2 & 4
	\end{tabular}
\end{center}
Hence for dice like this, we have
\begin{equation}
	\P{A < \tfrac{B + C}{2}} = \frac{13}{27} \approx 0.481.
	\label{thirteen_by_twenty-seven}
\end{equation}
So in this case, the casino is almost breaking even, while the player still has a small edge left.
If the value $4$ on the dice was a $3$ instead, then we would have a symmetric distribution again and the game would be even better for the player.
This underlines that in order to optimize the dice, the casino should aim for a certain skewness in the distribution.

It still remains unclear whether you are safe playing the game, or whether the casino can rig the dice in their favour.
In order to finally solve this problem, let's consider a slightly different game first.

\begin{game}[Beat the Average, card version]
    You pay \$1 to play.
    The croupier then shuffles a deck of cards with distinct values, draws three without replacement, and lets you choose one of them.
	You get \$2 if the value of your card is at least the average of the other two, and you lose otherwise.
\end{game}

Similar to the dice version, the casino can try to stack the deck in their favour by choosing the cards in a clever way.
For example, suppose that the deck contains only three cards---the minimum for the game to make sense---with values $1,2,4$.
Then there are only six possible outcomes in total, among which the winning combinations for the casino are:
\begin{center}
	\begin{tabular}{c|cccc}
		$A$ & 1 & 1 & 2 & 2 \\\hline
		$B$ & 2 & 4 & 1 & 4 \\\hline
		$C$ & 4 & 2 & 4 & 1
	\end{tabular}
\end{center}
Thus, in this setting, the player has a winning probability
$\P{A \ge \frac{B+C}{2}}$
of only $\frac{1}{3}$!
An excellent \emph{deal} for the casino.

For the card version, determining the winning probability for a given deck is a combinatorial problem.
If the deck consists of $m$ cards, then the game has
\[
	\binom{m}{3} \cdot 3
\]
possible outcomes, corresponding to the possible choices of three cards times the number of ways to assign one of them to the player.
So to determine the winning probability for the casino, we need to count how many of these combinations are winning.
How large can this number get?
Or as the player, at which bet value can you be confident that the game is in your favour, regardless of the deck used?
Here's the result.

\begin{prop}
	\label{prop:card_minimal}
	In the card version, the largest winning probability
        $\P{A < \frac{B+C}{2}}$
        for the casino is $\frac{2}{3}$, independently of the number of cards $m \ge 3$.
\end{prop}

\begin{proof}
	To achieve this number, we can use a deck with values generalizing the example above, namely
	\begin{equation}
		1, \: 2, \: 4, \: \dots, \: 2^{m-1}.
                \label{values_cards}
	\end{equation}
	To determine the number of winning combinations for the casino, note that for any triple of such numbers, the inequality $2 A < B + C$ is equivalent to $A < \max(B,C)$.
	There are exactly four ways that this holds for every unordered triple of cards, namely with the arrangements
	\[
		A < B < C, \qquad A < C < B, \qquad B < A < C, \qquad C < A < B.
	\]
	Therefore this choice of deck indeed achieves a winning probability of $\frac{2}{3}$.

	There is a neat trick, based on the probabilistic method~\cite{AS}, to show that the winning probability cannot be higher than $\frac{2}{3}$.
	Namely if we have a deck with $m$ cards, then for any $m' < m$, we can randomly pick a subset of $m'$ cards and discard the rest.
	In this way, we obtain a deck with $m'$ cards that we can play the game with.
	Now the winning probability for the original $m$-card deck is the expectation of the winning probabilities for all these decks with $m'$ many cards.
	Therefore, for \emph{at least one of these}, the winning probability must be at least as high as for the original deck.
        Since we know that the winning probability
        $\P{A < \frac{B+C}{2}}$
        cannot exceed $\frac{2}{3}$ for $m = 3$, using $m' = 3$ shows that this is still the case for any $m \ge 3$.

        An alternative (but essentially equivalent) argument is
        \begin{equation}
        3 \, \P{A < \tfrac{B+C}{2}}
        =
        \E{
            I\{A < \tfrac{B+C}{2} \}
            +
            I\{B < \tfrac{A+C}{2} \}
            +
            I\{C < \tfrac{A+B}{2} \}
        }
        \le 2,
        \label{probabilistic_method_indicators}
        \end{equation}
        where $I\{\cdot\}$ denotes the indicator function of the event in question.
	This equality holds because the joint distribution is invariant under permutations of the variables.
        The inequality holds because only two of the three indicator functions
        inside the expectation can be one
        simultaneously, as the sum of the three inequalities
        gives the contradiction $A + B + C < A + B + C$.
\end{proof}

In this way, we have 
solved the card version of the Beat the Average game: you should play if the casinos offers more than three times your bet if you win, while if they offer less then the casino can rip you off by making a clever choice of deck.

But how about the dice version?
We still don't know whether the casino can rig the dice in their favour!
Fortunately, having solved the card version will let us solve the dice version as well.
To note the relation, we may think of the dice version as a limiting case of the card version obtained by replacing the deck by a large number of copies of itself.\footnote{We can also consider the dice version as like the card version with replacement: a drawn card is immediately put back into the deck and may be drawn again. But we will not need this point of view.}
But more importantly, we can use the probabilistic method in a similar way as in the proof of \cref{prop:card_minimal} to solve the casino manager problem.

\begin{theorem}
	\label{theorem:dice_optimal}
        For any $m \ge 1$, the die with values \eqref{values_cards}
	and equal outcome probabilities achieves
        \begin{equation}
		\label{eq:dice_lower_bound}
		\P{A < \frac{B+C}{2}}
		= \frac{2}{3} - \frac{1}{2m} - \frac{1}{6m^2}.
	\end{equation}
	On the other hand, 
        every iid non-negative random variables $(A,B,C)$ satisfy
	\begin{equation}
		\label{eq:dice_upper_bound}
		\P{A < \frac{B+C}{2}} \le \frac{2}{3},
	\end{equation}
	and this inequality is strict 
        if $A$ has at least one atom,
        i.e., $\P{A=a}>0$ for some $a\in\R_+$.
\end{theorem}

\begin{proof}
        If $A=2^J$, $B=2^K$, and $C=2^L$ with $J,K,L$ independent and uniformly distributed in $\{0,\ldots,m-1\}$,
        then $A < \frac{B+C}{2}$ is equivalent to $J < \max(K,L)$, so that
        $$
        \P{A < \frac{B+C}{2}}
        = \sum_{j=1}^m \P{J=j} \, \P{\max(K,L) > j}
        = \sum_{j=1}^m \frac{1}{m} \left(1 - \left(\frac{j}{m}\right)^2\right).
        $$
        The identity $\sum_{j=1}^m j^2 = m(m+1)(2m+1)/6$ gives the equality
	in \eqref{eq:dice_lower_bound}.\footnote{We thank an anonymous referee for pointing out the precise value~\eqref{eq:dice_lower_bound}, which improves over our earlier lower bound.}
	This recovers for $m=2$
        the probability $3/8$ when $p=1/2$ in \eqref{three_by_eight},
        for $m=3$ the probability $13/27$ in 
        \eqref{thirteen_by_twenty-seven},
        and evaluates for $m=4$ to the probability $17/32$ in \eqref{seventeen_by_thirty-two} below.

        Inequality
        \eqref{probabilistic_method_indicators} holds for iid $(A,B,C)$,
        which proves \eqref{eq:dice_upper_bound}.

	It remains to be shown that the upper bound is not achieved when there is an atom.
	To this end, 
    	consider the three events corresponding to the number of distinct values,
    	\[
		E_1 \coloneqq \{|\{A,B,C\}|=1\},
                \quad
    		E_2 = \{ |\{A,B,C\}|=2 \},
                \quad
    		E_3 = \{ |\{A,B,C\}|=3 \}.
	\]
    The inequality $A<\frac{B+C}{2}$ obviously does not hold on $E_1$.
    Conditionally on $E_3$, the values $(A,B,C)$ are distributed
    as three cards drawn without replacement, and hence
    $\P{A < \tfrac{B+C}{2}\mid E_3} \le \frac23$
    by \Cref{prop:card_minimal}.
    Conditionally on $E_2$, we have
    $\P{A < \tfrac{B+C}{2}\mid E_2}
    = \frac12$. This implies that
    \begin{align*}
    \P{A < \tfrac{B+C}{2}}
    \le \frac{2}{3} \P{E_3} + \frac{1}{2} \P{E_2} 
    = \frac{2}{3} - \frac{1}{6} \P{E_2} - \frac{2}{3}\P{E_1}.
    \end{align*}
    Finally, 
    $\P{E_1} \ge \P{A=a}^3$ is strictly positive
    by $\P{A=a}>0$ if there is an atom $a$.
\end{proof}

So by the first statement of \cref{theorem:dice_optimal}, the casino can indeed rig the dice against you! 
The bound~\eqref{eq:dice_lower_bound} shows that this is possible with dice having $m \ge 4$ sides.
With a die having sides with values $1,2,4,8$, the winning probability for the casino is already
\begin{equation}
	\label{seventeen_by_thirty-two}
	\frac{17}{32} \approx 0.531.
\end{equation}
As the number of sides increases, the winning probability for the casino approaches $\frac{2}{3}$.
So you should be wary of playing, or at least demand a higher payout for when you win!

Indeed the card version of the game is fair
if the casino uses the cards \eqref{values_cards} and
offers three times your bet if you win, since in this case
the probability $\P{A<\frac{B+C}{2}}$ is exactly $\frac23$.
By \cref{theorem:dice_optimal}, such a fair game
with winning probability $\frac23$
cannot be achieved with dice with discrete outcomes.
It is still unclear if this can be achieved with dice with continuous outcomes, or whether the optimal probability of $\frac23$ cannot be achieved exactly.
We will return to the question of attainability
of the supremum in \Cref{section_open_problems}.

\begin{remark}
	Instead of~\eqref{eq:dice_lower_bound}, the following proof for the weaker bound
	\begin{equation}
		\label{eq:dice_lower_bound_weaker}
		\P{A < \frac{B+C}{2}} \ge \frac{2}{3} - \frac{3}{m}
	\end{equation}
	offers some additional insight into the relation between the dice and card versions of the game.

        Given a deck of $m$ cards with distinct values $1, 2, 4, \dots, 2^{m-1}$, we can consider a die with $m$ sides having the same values on the faces as on the cards, and such that all sides have equal probability $\frac{1}{m}$.
	Then if we exclude all rounds in which at least two of the dice land on the same side, the losing probability of the dice version trivially coincides with the losing probability of the card version.

	More formally, this amounts to considering the following joint distribution. 
        Throwing the die countably many times results in a sequence of outcomes $(X_t)_{t\ge 1}$, which are independent and uniformly distributed in $\{1,2,4,\dots,2^{m-1}\}$.
	Then let
	\[
		S \coloneqq \min\{s \ge 1 : X_s \neq X_1\}, 
		\qquad
		T \coloneqq \min\{t \ge 1 : X_t \neq X_1, X_S\}
	\]
	be the first times at which distinct outcomes are obtained.
	We take
	\begin{alignat*}{5}
		A & \coloneqq X_1, \qquad & B & \coloneqq X_2, \qquad & C & \coloneqq X_3, \\
		A' & \coloneqq X_1, & B' & \coloneqq X_S, & C' & \coloneqq X_T.
	\end{alignat*}
        Since $(A',B',C')$ has the same distribution as three cards
        drawn without replacement,
        $$
        \P{(A,B,C)=(A',B',C')}
        \ge \P{S=2, T=3} = \frac{(m-1)(m-2)}{m^2}
        \ge 1 - \frac{3}{m}.
        $$
        Since the cards $(A',B',C')$ drawn without replacement
        from the values in \eqref{values_cards}
        have the optimal probability
        $\P{A' < \tfrac{B' + C'}{2}} = \frac{2}{3}$,
	we obtain
        \[
		\P{A < \tfrac{B+C}{2}} \ge \P{A' < \tfrac{B' + C'}{2}} - \P{(A,B,C) \neq (A',B',C')}
		\ge \frac{2}{3} - \frac{3}{m},
	\]
	as was to be shown for the inequality in \eqref{eq:dice_lower_bound_weaker}.
\end{remark}

After enough customers have understood that the game is rigged against them by the casino having chosen the dice cleverly, 
the casino decides to change the game: Allow customers to
bring their own die and offer them a payout if
$A < \frac{B+C}{2}$ (instead of $A \ge \frac{B+C}{2}$ in previous games).

\begin{game}[Bring Your Own Die]
    The casino invites you to bring one die with outcomes in $\N$
    and arbitrary outcome probabilities.
    After you pay \$2 to play,
    the croupier rolls the die three times
    to get $A,B,C$.
    You win \$3 if
    $A < \frac{B+C}{2}$
    and lose otherwise.
\end{game}

Despite the apparent agency for customers in bringing their own die, and the possibility to get arbitrarily close to a fair game by choosing a suitable die, on average
the casino always has a strict advantage
by the non-achievability part of \Cref{theorem:dice_optimal}.
The player cannot argue that the 
payout should be strictly larger than $\$3$ either, since
any strictly larger payout would be unfair to, and could objectively bankrupt,
the casino.

\section{Optimizing over iid distributions in general}

As we have just seen, the casino manager's problem is an interesting and nontrivial mathematical question.
Let us now describe the general form of this kind of question and try to answer it.
Most of the developments in this section will be more general and more formal variations of the ideas we have already seen in the previous section.
As  will show, determining the optimal probability is much harder in general; the casino manager's problem has been a particularly simple case.

\begin{prob}
	\label{prob:optimizing_iid}
	Let $X$ be a measurable space and $f : X^n \to \R$ a bounded measurable function.
	Then find an algorithm to determine the quantity
	\begin{equation} 
		\label{eq:sup_mu}
		\sup_\mu \, \E[\mu^{\otimes n}]{f}.
	\end{equation}
	where $\mu$ runs over all probability measures on $X$ and $\mu^{\otimes n}$ denotes the $n$-fold product measure on $X^n$.
\end{prob}

This specializes to the casino manager's problem upon taking $X = \N$ and considering the indicator function of the event
\[
	E = \{(x_1, x_2, x_3) \in \N^3 \mid x_1 < \tfrac{x_2 + x_3}{2}\}.
\]
In order to generalize our solution of this problem into a general method, it will be useful to have a simple notation for averaging over a finite set.

\begin{nota}
	For a finite set $A$ and $f : A \to \R$, we write
	\[
		\int_{x \in A} f(x) \coloneqq \frac{1}{|A|} \sum_{x \in A} f(x).
	\]
\end{nota}

That is, when no measure is specified for an integral, then we leave it understood that we are averaging uniformly over a finite set.
This also applies to averaging over maps: with $B^A$ denoting the set of maps $A \to B$ for finite sets $A$ and $B$, for any $\Phi : B^A \to \R$ we write
\[
	\int_{g : A \to B} \Phi(g) \coloneqq \frac{1}{|B|^{|A|}} \sum_{g : A \to B} \Phi(g).
\]
If $|A| \le |B|$, then it is also meaningful to average over all injections $i : A \hookrightarrow B$, which we denote by
\[
	\int_{i : A \hookrightarrow B} \Phi(i) \coloneqq \frac{1}{|B| (|B| - 1) \cdots (|B| - |A| + 1)} \sum_{i : A \hookrightarrow B} \Phi(i).
\]
Given a third set $C$ with $|B| \le |C|$, we can average over injections $A \hookrightarrow C$ by averaging over injections $A \hookrightarrow B$ and injections $B \hookrightarrow C$ separately. In other words, 
\begin{equation}
	\label{inj_avg_comp}
	\int_{i : A \hookrightarrow C} \Phi(i) \,=\, \int_{j : A \hookrightarrow B} \int_{k : B \hookrightarrow C} \Phi(k \circ j).
\end{equation}
This equation will come in handy in our proofs.

Let us start the general development by deriving an upper bound on~\eqref{eq:sup_mu}, the quantity we are interested in.
As far as we know, this result was first stated (for Polish spaces) in a 2019 MathOverflow answer by Daniel Lacker,\footnote{See \href{https://mathoverflow.net/questions/323302/maximizing-the-expectation-of-a-polynomial-function-of-iid-random-variables/323842\#323842}{mathoverflow.net/questions/323302/maximizing-the-expectation-of-a-polynomial-function-of-iid-random-variables/323842\#323842}.} before we rediscovered it independently.

\begin{lemma}
	\label{sup_mu_ineq}
	For every $m \ge n$, we have
	\begin{equation}
		\label{eq:sup_mu_ineq}
		\sup_\mu \E[\mu^{\otimes n}]{f} \le \sup_{x_1, \dots, x_m \in X} \int_{i : [n] \hookrightarrow [m]} f(x_{i(1)}, \dots, x_{i(n)}).
	\end{equation}

\end{lemma}

\begin{proof}
	For fixed $\mu$ and every $i : [n] \hookrightarrow [m]$, we have by exchangeability
	\begin{align*}
		\E[\mu^{\otimes n}]{f} & = \int_{x_1, \ldots, x_n \in X} f(x_1, \ldots, x_n) \, \mu^{\otimes n}(\mathrm{d}x) \\
			& = \int_{x_1, \ldots, x_m \in X} f(x_{i(1)}, \ldots, x_{i(n)}) \, \mu^{\otimes m}(\mathrm{d}x).
	\end{align*}
	This equality still holds if $i$ itself is chosen uniformly at random from all injections $[n] \hookrightarrow [m]$.
	This gives
	\begin{equation}
		\label{eq:inj_avg}
		\E[\mu^{\otimes n}]{f} = \int_{x_1, \ldots, x_m \in X} \int_{i : [n] \hookrightarrow [m]} f(x_{i(1)}, \ldots, x_{i(n)}) \, \mu^{\otimes m}(\mathrm{d}x).
	\end{equation}
	The claim now follows from the fact that an integral with respect to a probability measure is bounded by the supremum of the integrand.
\end{proof}

As we show next, this upper bound can only get better as the free parameter $m$ increases, and in fact it converges to the exact value we are interested in.

\begin{theorem}
	\label{sup_mu}
	Let $X$ be a measurable space and $f : X^n \to \R$ bounded measurable.
	Then
	\[
		\sup_\mu \E[\mu^{\otimes n}]{f} = \lim_{m \to \infty} \sup_{x_1, \dots, x_m \in X} \int_{i : [n] \hookrightarrow [m]} f(x_{i(1)}, \dots, x_{i(n)}),
	\]
	where the limit is over a monotonically nonincreasing sequence.
\end{theorem}

So intuitively, maximizing an expectation value over all iid distributions is equivalent to maximizing it over all distributions that correspond to sampling without replacement from an $m$-element set and taking $m \to \infty$.

The following proof is essentially a simplified presentation of Lacker's.
The second part, which constructs a sequence of distributions $\mu$ converging to the supremum, was suggested to us by Will Sawin.\footnote{See the discussion at \href{https://mathoverflow.net/questions/474916/how-large-can-mathbfpx-1-x-2-x-3-2-x-4-get/\#comment1234501_475013}{mathoverflow.net/questions/474916/how-large-can-mathbfpx-1-x-2-x-3-2-x-4-get/\#comment1234501\_475013}.} 
It can be understood as an application of Freedman's bound on the total variation distance between sampling with and without replacement~\cite{freedman}.

\begin{proof}
	For the monotonicity in $m$, consider $m' \ge m$.
	Then using~\eqref{inj_avg_comp}, the averaging over injections $[n] \hookrightarrow [m']$ can be achieved by separately averaging over injections $[n] \hookrightarrow [m]$ and injections $[m] \hookrightarrow [m']$, which gives
	\begin{align*}
		\sup_{x_1, \dots, x_{m'} \in X} & \int_{i : [n] \hookrightarrow [m']} f(x_{i(1)}, \dots, x_{i(n)}) \\
		& = \sup_{x_1, \dots, x_{m'} \in X} \int_{j : [n] \hookrightarrow [m]} \int_{k : [m] \hookrightarrow [m']} f(x_{k(j(1))}, \dots, x_{k(j(n))}) \\
		& \le \sup_{x_1, \dots, x_{m'} \in X} \sup_{k : [m] \hookrightarrow [m']} \int_{j : [n] \hookrightarrow [m]} f(x_{k(j(1))}, \dots, x_{k(j(n))}) \\
		& = \sup_{x_1, \dots, x_m \in X} \int_{j : [n] \hookrightarrow [m]} f(x_{j(1)}, \dots, x_{j(n)}),
	\end{align*}
	where the inequality step is once again thanks to the fact that an expectation is bounded by the supremum,
	and the final step holds for trivial reasons, as both sides conduct in effect the same optimization.

	For the convergence, let $x_1, \dots, x_m \in X$ be given and consider the uniform distribution over these values,
	\begin{equation}
		\label{eq:mu}
		\mu \coloneqq \sum_{i=1}^m \frac{1}{m} \delta_{x_i}.
	\end{equation}
	The probability that a uniformly random map $k : [n] \to [m]$ is injective is
	\begin{equation}
		\label{eq:distinct_prob}
		\frac{m-1}{m} \cdot \ldots \cdot \frac{m-n+1}{m} \ge \left(1 - \frac{n}{m}\right)^n.
	\end{equation}
	Therefore, assuming $f \ge 0$ without loss of generality, we can bound
	\[
		\E[\mu^{\otimes n}]{f} = \int_{k : [n] \to [m]} f(x_{k(1)}, \dots, x_{k(n)}) 
		\ge \left(1 - \frac{n}{m}\right)^n \int_{i : [n] \hookrightarrow [m]} f(x_{i(1)}, \dots, x_{i(n)}).
	\]
	Hence the upper bound~\eqref{eq:sup_mu_ineq} coincides with the actual value up to a factor of $\left( 1 - \frac{n}{m} \right)^n$, which converges to $1$ as $m \to \infty$.
\end{proof}

\begin{remark}
	\label{lower_bounds}
	As the proof shows, the sequence of upper bounds does not only converge to the actual value, but we also get a bound on how far we are from the actual value, namely for $f \ge 0$ we are off by at most a factor of
	\[
		\frac{m-1}{m} \cdot \ldots \cdot \frac{m-n+1}{m} \ge \left(1 - \frac{n}{m}\right)^n.
	\]
	Also for every $m$, by~\eqref{eq:mu} there is a measure supported on $m$ points that achieves the supremum up to such a factor.
\end{remark}

The following application is due to Alon and Yuster, and our proof is essentially theirs formulated in our language.

\begin{ex}[{123 Theorem~\cite{AY}}]
	We prove that whenever $X_1$ and $X_2$ are real iid random variables, then
	\begin{equation}
		\label{eq:iid_diff_bound}
		\P{|X_1 - X_2| \le 2} \le 3 \, \P{|X_1 - X_2| \le 1}.
	\end{equation}
	In fact, Alon and Yuster prove this with strict inequality, and in a more general form where the bounds on the distance are arbitrary~\cite[Theorem~1.2]{AY}. 
	We focus on this special case for simplicity.

	To this end, we apply \Cref{sup_mu} with $n = 2$ and the function 
	\begin{align*}
		f(x_1, x_2) & \coloneqq I\{|x_1 - x_2| \le 2\} - 3 \, I\{|x_1 - x_2| \le 1\} \\[2pt]
			    & \phantom{:}= I\{|x_1 - x_2| \in (1,2]\} - 2 \, I\{|x_1 - x_2| \le 1\},
	\end{align*}
	and we show that
	\begin{equation}
		\label{eq:iid_diff_bound_2}
		\sup_{x_1, \dots, x_m \in \R} \sum_{1 \le i < j \le m} f(x_i, x_j) \le 2m - 2.
	\end{equation}
	This implies the claim upon normalizing by $\binom{m}{2}$ and taking $m \to \infty$, since then the right-hand side tends to zero.

	We thus need to prove that for every tuple of real numbers $(x_1, \dots, x_m)$, the number of pairs $i < j$ with $|x_i - x_j| \in (1,2]$ minus twice the number of pairs with $|x_i - x_j| \le 1$ is at most $2m - 2$.
	This is easy to see by induction on $m$, where the base case $m = 1$ is trivial.
	If we let $j$ be an index such that $x_j$ has the most other numbers within distance $\le 1$, then we can apply the induction hypothesis to the shorter tuple with $x_j$ removed.
	To see that adding $x_j$ can increase the total count by at most $2$, we need to show that there are at most twice as many other numbers at a distance within $(1, 2]$ from $x_j$ than there are at a distance $\le 1$ except for two.
	And indeed we have
	\[
		| \{ i \neq j \mid 1 < x_i - x_j \le 2 \} | \: \le \: | \{ i \neq j \mid -1 \le x_i - x_j \le 1 \} | + 1,
	\]
	since otherwise we would have a number in $(x_j + 1, x_j + 2]$ which has more other numbers within distance $\le 1$ than $x_j$ does.
	Using the same argument for $-2 \le x_i - x_j < -1$ completes the induction step and therefore the proof of~\eqref{eq:iid_diff_bound}.

	We can now consider sequences for which the left-hand side of~\eqref{eq:iid_diff_bound_2} is large and take the uniform distribution over them in order to get distributions for which~\eqref{eq:iid_diff_bound} is tight, as we did in the proof of \Cref{sup_mu}.
	While we do not know whether the bound of $2m - 2$ can be achieved, the tuple defined by $x_i \coloneqq c i$
	for any $c \in (1, 2]$ achieves $m - 1$.
	This amounts to the same in the limit as $m \to \infty$ since both expressions are $o(m^2)$.
	Therefore taking
	\[
		\mu \coloneqq \frac{1}{m} \sum_{i=1}^m \delta_{c i}
	\]
	produces a sequence of distributions for which~\eqref{eq:iid_diff_bound} is tight, in the sense that no positive constant can be subtracted from the right-hand side without violating the inequality.
	Calculating
	\[
		\P[\mu^{\otimes 2}]{|X_1 - X_2| \le 1} = \P[\mu^{\otimes 2}]{X_1 = X_2} = \frac{1}{m}, \qquad
		\P[\mu^{\otimes 2}]{|X_1 - X_2| \le 2} = \frac{3m - 2}{m^2},
	\]
	shows that the inequality is tight also in a different sense, namely in that the factor of $3$ in~\eqref{eq:iid_diff_bound} cannot be improved.
\end{ex}

\section{Maximizing the probability of a strict inequality}

Let us consider more concretely the problem of maximizing the probability of a strict inequality
\begin{equation}
	\label{eq:strict_ineq}
	\sum_{i=1}^n c_i X_i > 0
\end{equation}
to hold, where the coefficients $c_i \in \R$ are fixed and the $X_i$ are iid real-valued random variables.
That is, we try to determine
\begin{equation}
	\label{q:sup_mu_ineq}
	\sup_{\mu} \, \P[\mu^{\otimes n}]{\sum_{i=1}^n c_i X_i > 0} = \; ?
\end{equation}
where $\mu$ ranges over all probability measures on $\R$.
As already noted in \Cref{beat_the_average}, it is essential to use strict inequality, since with non-strict inequality we can trivially achieve probability $1$ by taking $X_1 = \ldots = X_n = 0$ deterministically.
In addition, the problem is interesting only if $\sum_{i=1}^n c_i = 0$, since otherwise~\eqref{eq:strict_ineq} also has a deterministic solution given by $X_1 = \ldots = X_n = 1$ if $\sum_{i=1}^n c_i > 0$ and by $X_1 = \ldots = X_n = -1$ if $\sum_{i=1}^n c_i < 0$.

As a variation on this theme, we can also consider the case where the $X_i$ are nonnegative, or equivalently where $\mu$ is supported on $\R_+$.
This is the setting of the Beat the Average game, and the general problem that we will focus on in the following.
As per the above, this problem is nontrivial as soon as $\sum_{i=1}^n c_i \le 0$, which we assume to be the case.
For example, the casino manager's problem is concerned with the inequality
\[
	X_1 + X_2 - 2 X_3 > 0.
\]
In this case, we had found that the sequence of bounds of \Cref{sup_mu} is constant in $m$, and therefore already tight at $m = n = 3$.
In \Cref{my_example} below, which is motivated by the law of large numbers, we will see that this is not the case in general.

\begin{remark}
	\label{rem:tension}
	There is a peculiar tension which makes problems of the form~\eqref{q:sup_mu_ineq} seem particularly interesting: on the one hand, a distribution that is close to optimal cannot be supported away from $0$, since otherwise one could improve the probability of~\eqref{eq:strict_ineq} by shifting the distribution towards the left; on the one hand, it cannot have too much weight on $0$, since the probability of $X_1 = \ldots = X_n = 0$ cannot be too high.
\end{remark}

\begin{ex}
	\label{my_example}
	The solution to the casino manager's problem, which we considered in \Cref{beat_the_average}, is
	\begin{equation}
		\label{eq:casino_solution}
		\sup_\mu \, \P[\mu^{\otimes 3}]{2 X_1 < X_2 + X_3} = \frac{2}{3},
	\end{equation}
	where $\mu$ ranges over all probability measures on $\N$ or $\R_+$.\footnote{Note that our derivation from \Cref{beat_the_average} applies either way.}

	This problem can be seen as an instance of the following: 
        determine the value of
	\begin{equation}
		\label{eq:gen_example}
		\sup_\mu \, \P[\mu^{\otimes n}]{\frac{1}{n} \sum_{i=1}^n X_i > \frac{\alpha}{m} \sum_{i=1}^m X_i},
	\end{equation}
        for given $n > m$ and $\alpha > 0$, 
	where $\mu$ ranges over all probability measures on $\R_+$?
	This question is motivated by the law of large numbers: how strongly can the sample average grow with the number of samples?
	Surprisingly, there are nontrivial universal bounds on these quantities, even without assuming that $\mu$ has moments of any order.
	Our~\eqref{eq:casino_solution} is exactly the $n = 3$, $m = 1$ and $\alpha = 1$ case of this.
\end{ex}

In order to address~\eqref{q:sup_mu_ineq} in general, and with the support restricted to $\R_+$, we can use \Cref{sup_mu_ineq} to derive upper bounds.
For fixed $m \ge n$ and a given injection $\alpha : [n] \hookrightarrow [m]$, let us call
\[
	\sum_{i=1}^n c_i x_{\alpha(i)} > 0
\]
a \newterm{version} of the inequality under consideration.
Thus the total number of versions of the inequality is equal to the falling factorial $m (m-1) \cdots (m-n+1)$.\footnote{If some of the coefficients $c_i$ coincide, then there are additional symmetries which effectively reduce the number of versions, which happens for example for the inequalities of \Cref{my_example}.}
Then \Cref{sup_mu_ineq} states that~\eqref{eq:sup_mu_ineq} is upper bounded by the largest number of versions that are jointly satisfied for any deterministic assignment $x_1, \ldots, x_m \in \R_+$ divided by the total number of versions.
Calculating this upper bound is an instance of the \newterm{maximum feasible subsystem} problem with strict linear inequalities, which is NP-hard in general~\cite[Theorem~4]{AK}.\footnote{It is conceivable that the instances of the maximum feasible subsystem problem which arise in our context are easier. The fact that all constraints coincide up to permutations of the variables equips these instances with additional structure that can possibly be exploited.}

As for lower bounds, \Cref{lower_bounds} lets us turn these upper bounds into concrete $\mu$'s which provide lower bounds.
But in the present setting of linear inequalities, the following 
proposition provides better lower bounds.

\begin{prop}
	\label{strict_ineq_geom}
	Let $\nu$ be a probability measure on $\R_+$
        with $\P[\nu^{\otimes n}]{\sum_{i=1}^n c_i X_i > 0} > 0$.
	Then for every $\eps > 0$, 
	there is a probability measure $\mu$ on $[0,1]$ such that
	\begin{equation}
		\P[\mu^{\otimes n}]{\sum_{i=1}^n c_i X_i > 0} \ge \frac{\P[\nu^{\otimes n}]{\sum_{i=1}^n c_i X_i > 0}}{1 - \P[\nu^{\otimes n}]{\sum_{i=1}^n c_i X_i = 0}} - \eps.
                \label{conclusion_strict}
	\end{equation}
        If furthermore $\nu$ is discrete with finitely many atoms, then there 
        exists a measure $\mu$ satisfying \eqref{conclusion_strict} with $\eps=0$.
\end{prop}

\begin{proof}
    Assume first that $\nu$ is supported
    in $[0,A]$ for some $A>0$.
    Let $\eta\in(0,1)$
    and define
    $Y_i \coloneqq \sum_{j=1}^\infty \eta^{j-1} X_{i,j}$
    where $(X_{i,j})_{i\in[n], j\in\N}$ are iid with distribution $\nu$ and the assumed boundedness ensures that the series converges almost surely.\footnote{We owe this concise formulation to Will Sawin. See the comment \href{https://mathoverflow.net/questions/474916/how-large-can-mathbfpx-1-x-2-x-3-2-x-4-get/475013\#comment1234202\_474927}{mathoverflow.net/questions/474916/how-large-can-mathbfpx-1-x-2-x-3-2-x-4-get/475013\#comment1234202\_474927}.}
    Similarly, set
    $Y_i' \coloneqq \sum_{j=2}^\infty \eta^{j-2} X_{i,j}$
    so that $Y_i = X_{i,1} + \eta Y_i'$
    and $Y_i'$ is equal in distribution to $Y_i$.
    Define
    $S = \sum_{i=1}^n c_i Y_i$
    and $S' = \sum_{i=1}^n c_i Y_i'$,
    so that $S= \sum_{i=1}^n c_i X_{i,1} + \eta S'$.
    Then with $I\{\cdot\}$ denoting the indicator function of an event, we have almost surely
    $$
    I\Bigl\{ S > 0 \Bigr\}
    -
    I\Bigl\{ \sum_{i=1}^n c_i X_{i,1} = 0
\Bigr\}
I\Bigl\{ S' > 0 \Bigr\}
    \ge 
    I\Bigl\{ \sum_{i=1}^n c_i X_{i,1} >
        \max\Bigl(0,
    - \eta S'
        \Bigr)\Bigr\}.
    $$
    Taking expectations, the left-hand side
    equals $\P{S>0}(1 - \P{\sum_{i=1}^n c_i X_{i,1} = 0})$
    since $S$ and $S'$ are equal in distribution 
    and by independence for the negative term.
    If $\nu$ is discrete with finitely many atoms
    then so does the law of $\sum_{i=1}^n c_i X_{i,1}$,
    and one can find $\eta>0$
    small enough so that the expectation of the right-hand side
    is exactly $\P{\sum_{i=1}^n c_i X_{i,1} > 0}$.
    For general $\nu$, the limit as $\eta\to 0$ of the expectation of the
    right-hand side is $\P{\sum_{i=1}^n c_i X_{i,1} > 0}$, for instance by the
    monotone convergence theorem.
    With $\mu$ the law of $Y_i$, this completes
    the proof if $\nu$ is supported in $[0,A]$.

    In general,
    let $A>0$ be large enough so that
    $\P{\exists i\in[n], \min(A, X_i) \ne X_i}\le \eps$
    for iid $X_i\sim \nu$.
    Applying the previous argument to the distribution $\nu_A$
    of $\min(A,X_i)$, 
    there exists $\mu$ such that \eqref{conclusion_strict}
    holds with $\nu_A$ in place of $\nu$.
    Since the numerator (resp.~denominator) for $\nu$ and $\nu_A$
    are at most $\eps$ apart, the conclusion follows modulo a reparametrization of $\eps$.
\end{proof}

	\section{The case of 
        	\texorpdfstring{$\sup_{\mu} \, \P[\mu^{\otimes 4}]{X_1 + X_2 + X_3 < 2 X_4}$}{sup P[X1+X2+X3 < 2 X4]}
        }
	\label{example}

	While determining~\eqref{eq:gen_example} in general seems to be challenging open problem, let us consider one further instance of it.
        For $n = 3$, $m = 4$ and $\alpha = 9/8$, the problem \eqref{eq:gen_example} is equivalent to maximizing the probability of $X_1 + X_2 + X_3 < 2 X_4$ for nonnegative iid variables.
	We now use the results obtained so far  to derive lower and upper bounds for this quantity, namely
	\begin{equation}
            \boxed{
            0.400695 \,\le\, \sup_{\mu} \P[\mu^{\otimes 4}]{X_1 + X_2 + X_3 < 2 X_4} \,\le\,
	    \frac{673}{1615} \,\le\, 0.417 
        }
        \label{4_eq_1}
	\end{equation}
        where $\mu$ ranges over all probability measures on $\R_+$.
	The precise value remains unknown.

        \subsection{Lower bound}

	We now sketch the path that we took towards the lower bound in \eqref{4_eq_1} before presenting the statement and proof as \Cref{prop:lower_bound_04}.

        Proving a lower bound requires the construction
        of a measure $\mu$ with large probability for the event
        $X_1 + X_2 + X_3 < 2X_4$.
	Let us consider finitely supported $\mu$ for the moment, so that $\mu$ is defined by specifying the locations $a_i$ and weights $\mu(\{a_i\})$ of its atoms, and make the tension noted in \cref{rem:tension} more precise.
        It is quite clear that $0$ should be an atom, because otherwise, subtracting the smallest atom from $X_i$ would yield new variables 
        with a higher probability for the strict inequality
        in \eqref{4_eq_1}.
        Then, the presence of an atom at $0$ implies that the event
        $X_1 + X_2 + X_3 = 2X_4$ also has positive probability, namely
        \[
		\P[\mu^{\otimes 4}]{X_1 + X_2 + X_3 = 2X_4} \ge 3 \mu(\{0\}) \sum_{i \: : \: a_i \neq 0} \mu(\{a_i\})^3,
	\]
	since the desired equality holds as soon as one term on the left is zero and the three other variables coincide.
        The utility of \Cref{strict_ineq_geom} is now that it lets us reallocate
        some of the probability mass that is ``lost'' to $\P{X_1 + X_2 + X_3 = 2X_4}$ in the presence of an atom at $0$ to the desired event $X_1 + X_2 + X_3 < 2X_4$.
	We can in particular conclude that a finitely supported distribution cannot be optimal.

	For example, \Cref{strict_ineq_geom} applied to 
	$\nu=\mathrm{Bernoulli}(p)$ and $\eps=0$
        and optimizing over $p\in(0,1)$ yields that the 
        lower bound \eqref{conclusion_strict}, namely
	\begin{equation}
                \P[\mu^{\otimes 4}]{X_1+X_2+X_3  < 2X_4} \ge \frac{\P[\nu^{\otimes 4}]{X_1+X_2+X_3  < 2X_4}}{1 - \P[\nu^{\otimes 4}]{X_1+X_2+X_3  = 2X_4}}
                \label{conclusion_strict_no_eps}
	\end{equation}
        is at least $0.343$ at $p\approx 0.404$.
        A better lower bound can be obtained with three atoms for $\nu$.
        With further numerical experimentation,
	taking $
		\nu \coloneqq \frac{1}{2} \delta_0 + \frac{1}{6} \delta_5 + \frac{1}{3} \delta_9$ gives\footnote{For the second probability, note that $0 + 0 + 0 = 2 \cdot 0$ is the only solution to $X_1 + X_2 + X_3 = 2 X_4$ with $X_i \in \{0,5,9\}$. For the inequality, the solutions are the following, with braces denoting either possibility: \, $0 + 0 + 0 < 2 \cdot \{5, 9\}$; \, $0 + 0 + \{5, 9\} < 2 \cdot \{5, 9\}$; \, $0 + 5 + 5 < 2 \cdot 9$; \, $0 + 5 + 9 \le 2 \cdot 9$; \, $5 + 5 + 5 < 2 \cdot 9$; \, and their permutations.}
	\[
		\P[\mu^{\otimes 4}]{X_1 + X_2 + X_3 < 2 X_4} = \frac{26}{81}, \qquad \P[\mu^{\otimes 4}]{X_1 + X_2 + X_3 = 2 X_4} = \frac{1}{8}.
	\]
	Hence in combination with \Cref{strict_ineq_geom}, this distribution provides
        a lower bound \eqref{conclusion_strict_no_eps}
        equal to $\frac{208}{567}\approx 0.367$.

        Additional numerical experimentation
        involving slowly increasing the number of atoms,
        and for each new support optimizing the probability mass
        function using gradient ascent, led to 
        an improved lower bound of $0.391$.
        A pattern emerged from optimizing the weights of
        the probability mass functions over
        these atoms to maximize \eqref{conclusion_strict_no_eps}
        with gradient descent:
	good distributions were approximately of the form
	\[
		\nu = \frac{1}{2} \delta_0 + \frac{1-q}{2N} \sum_{i=1}^N \delta_{1-2^{-i}} + \frac{q}{2} \delta_{1-2^{-(N+1)}}
	\]
	with parameters $N \in \N$ and $q\in (0,1)$.
        The values $N=61$ and $q=0.0546388$ achieve
        $0.398$ in \eqref{conclusion_strict_no_eps}.
	In this way, we arrived at the following lower bound, where also the weight of the atom at $0$ is optimized over.

        \begin{prop}[Lower bound]
            \label{prop:lower_bound_04}
            For every $\eps > 0$ there exists a discrete measure $\mu$ supported on $(0,1)$ such that
            \begin{equation}
		    \label{eq:lower_bound}
		    \P[\mu^{\otimes 4}]{X_1 + X_2 + X_3 < 2 X_4} \ge \sup_{p \in (0,1)} \frac{p(2-p)}{1 + p + p^2 + p^3} - \eps.
	    \end{equation}
        \end{prop}

	By taking $p = 0.474346$, which is the approximate value of $p$ that achieves the supremum on the right, we obtain the numerical lower bound in \eqref{4_eq_1}.\footnote{Our original lower bound of $0.4$ in an earlier version of this paper only considered $p = \frac{1}{2}$ in~\eqref{eq:mu_p} for the weight of the atom at $0$, which gives $\frac{2}{5}-\eps$ in~\eqref{eq:lower_bound}. We thank Vincent Yu for pointing out that $p=\frac{1}{2}$ was not the optimal choice (see \href{https://mathoverflow.net/a/484853/27013}{mathoverflow.net/a/484853/27013}).}

        \begin{proof}
        We choose $N \in \N$, for which we will take $N \to \infty$ at the end.
	In terms of this, define the discrete measure $\nu$ as
	\begin{equation}
		\label{eq:mu_p}
		\nu \coloneqq p \delta_0 + \frac{1-p}{N} \sum_{i=1}^N \delta_{1-2^{-i}}.
	\end{equation}
        Then the probability of strict inequality is bounded from below
        as
	\[
        \P{X_1+X_2+X_3 < 2X_4}
        \ge 
        \begin{cases}
            & \!\!\! \P{X_1=0, X_2=0, X_3 = 0, X_4 > 0}
            \\
            + \, 3 & \!\!\! \P{X_1=0, X_2=0, X_3 > 0, X_4 > 0, ~~X_3 < 2X_4}
            \\
            + \, 3 & \!\!\! \P{X_1=0, X_2 >0, X_3 >0, X_4 >0,  ~~X_2+X_3 < 2X_4}
        \end{cases}
	\]
        since the events in the terms on the right-hand side are disjoint,
        each implies $X_1+X_2+X_3 < 2X_4$, and the factors of $3$ stem from the possibility to permute the variables.
        The event on the first line
        has probability $p^3(1-p)$.
        Furthermore since
        $2X_4\ge 1$ as soon as $X_4\ne 0$ and 
        $X_3 <1$ always holds, the second line equals
        \[
		3 \, \P{X_1=X_2=0, X_3 > 0, X_4>0} = 3 p^2 (1-p)^2.
	\]
        By the definition of conditional probability,
        the third line can be written as
	\[
        3 \, \P{X_1=0, X_2 >0, X_3 >0, X_4 >0}
         \: \P{X_2 + X_3 < 2X_4 \mid X_1 = 0, X_2 >0, X_3 >0, X_4 >0}.
	\]
        If $X_2,X_3, X_4$ are positive,
        then they all belong to $\{1-\frac{1}{2^i} \: : \: i=1,\dots,N\}$
        and the implication
        \begin{equation}
        X_3 < X_4
	\; \Rightarrow \;
        X_2 + X_3 < 2X_4
        \label{impl}
        \end{equation}
        holds thanks to the following argument.
	Since the values of all variables are of the form $1-\frac{1}{2^i}$, 
        it is easy to see that $X_3 < X_4$ implies\footnote{This is seen most easily by considering the $Y_i \coloneqq 1 - X_i$.}
	$1 + X_3 \le 2 X_4$. 
        Since $X_2 < 1$, the implication 
        \eqref{impl} follows.
        Similarly, by exchanging the roles of $X_2$ and $X_3$,
	the implication
        \begin{equation}
            \label{impl2}
            X_2 < X_4 
	    \; \Rightarrow \;
	    X_2+X_3<2X_4
        \end{equation}
        also holds.
        Denoting by $E \coloneqq \{X_1 = 0, X_2 >0, X_3 >0, X_4 >0\}$
	the conditioning event for brevity, we clearly have $\P{E} = p (1 - p)^3$.
	By the earlier considerations, we also have
        \begin{align*}
         &\P{X_2 + X_3 < 2X_4 \mid E}
         \\&\ge \P{
             \min(X_2,X_3) < X_4
            \mid E
	    } && \text{[by implications \eqref{impl}--\eqref{impl2}]} \\
         &\ge \P{
		 \min(X_2,X_3) < X_4 \text{ and }
             X_2,X_3,X_4\text{ distinct}
            \mid E
         }
	 && \text{[inclusion]}
       \\& = \tfrac{2}{3} \, \P{
             X_2,X_3,X_4\text{ distinct}
            \mid E
         }
	   && \text{[by symmetry]}
	 \\& = \tfrac{2}{3} \, N(N-1)(N-2)/N^3
	   && \text{[counting distinct triples]}
         .
        \end{align*}
        In summary, we have established the lower bound
        \begin{align*}
        \P{X_1+X_2+X_3 < 2X_4}
	&\ge p^3 (1 - p) + 3 p^2 (1 - p)^2 + 2 p (1 - p)^3 \frac{N(N-1)(N-2)}{N^3}
      	\\[4pt]
	& = p (1 - p) (2 - p) - O(N^{-1}).
        \end{align*}
        For the probability of equality appearing in the denominator
        of \eqref{conclusion_strict}, we simply use
        $$
        \P{X_1+X_2+X_3 = 2X_4}
	\: \ge \: \P{X_1=X_2=X_3=X_4=0} = p^4.
        $$
        Now \Cref{strict_ineq_geom} with $\eps=0$ yields the existence of a measure $\mu$ on $[0,1]$ such that
	\begin{align*}
        \P[\mu^{\otimes 4}]{X_1 + X_2 + X_3 < 2 X_4} 
	& \ge 
	\frac{p(1-p)(2-p)}{1 - p^4} - O(N^{-1}) \\
	& =
	\frac{p(2-p)}{1 + p + p^2 + p^3} - O(N^{-1}).
        \end{align*}
        Taking $N\to\infty$ completes the proof
        of the lower bound in \eqref{4_eq_1}.
        \end{proof}

	\begin{remark}
		Optimizing the rational function on the right-hand side of~\eqref{eq:lower_bound} over $p$ can be done analytically.
		Indeed the maximum is the unique solution in $(0,1)$ of the fourth order equation
		\begin{equation}
			\label{4th_order_equation}
			2 - 2p - 3 p^2 - 4 p^3 + p^4 = 0,
		\end{equation}
		which can be expressed analytically in terms of radicals.
		Therefore we would also obtain a lower bound in~\eqref{4_eq_1} that is exact and slightly improves on the numerical value of $0.400695$ by trailing digits.
		However, calculating these explicit expressions in \texttt{SageMath} shows that they are a bit too large to be displayed here, and they do not seem to provide any additional insight.
	\end{remark}

	\begin{remark}
            The proof of \Cref{prop:lower_bound_04} derives
            a lower bound for each value of $p$ and $N$, takes
            the limit as $N\to+\infty$ and then optimizes over $p$.
            A referee pointed out that one can also evaluate
            $\P[\nu^{\otimes n}]{X_1 + X_2 + X_3 < 2 X_4}/(1-p^4)$
            exactly for each value of $p$ and $N$ using symbolic 
            computational tools such as
	    \texttt{Mathematica} or \texttt{SymPy}, and optimize over $p$ for a fixed
            value of $N$. Performing this operation for $N=1,\dots,150$
            gives values closer and closer to $2/5$, but eventually
            $N=156$ is the first integer providing a probability
            strictly greater than $2/5$.
	    Thus this is an instructive example of how numerical 
	    experimentation can be misleading if not taken far enough.
        \end{remark}

	\subsection{Upper bound} 

	Before proving the upper bound given in \eqref{4_eq_1}, it is instructive to first derive weaker bounds.
	To this end, we apply \Cref{sup_mu_ineq} first with $m = 4$ and then with $m = 6$; this is interesting insofar as it provides an example where the upper bound improves with increasing $m$.
	The bound for $m = 6$, which will be $\frac{7}{15} \le 0.467$, is also the best upper bound that we know of which can be verified by hand.
	In the next subsection, we will then present the computer-assisted proof of the upper bound in \eqref{4_eq_1}.

	For $m = 4$, the four versions of the inequality are
	\[
		\begin{aligned}
			& X_1 + X_2 + X_3 < 2 X_4, 
			\qquad& X_1 + X_2 + X_4 < 2 X_3, \\
			& X_1 + X_3 + X_4 < 2 X_2, 
			\qquad& X_2 + X_3 + X_4 < 2 X_1.
		\end{aligned}
	\]
	To apply \Cref{sup_mu_ineq}, we need to determine how many of these are jointly satisfiable.
	Assuming $X_1 \le X_2 \le X_3 \le X_4$ without loss of generality, it is clear that at most the first two are jointly satisfiable, and this is possible with
	$X_1 = X_2 = 0$ and $X_3 = X_4 = 1$.
	Therefore by~\Cref{sup_mu_ineq}, we can conclude
	\begin{equation}
		\label{eq:m4}
		\sup_\mu \P[\mu^{\otimes 4}]{X_1 + X_2 + X_3 < 2 X_4} \le \frac{2}{4} = \frac{1}{2}.
	\end{equation}

	For general $m$, there are
	\begin{equation}
            \textstyle
		4 \cdot \binom{m}{4} = (m - 3) \cdot \binom{m}{3}
                \label{number_of_satisfiable_inequalities}
	\end{equation}
	different versions of the inequality 
	\begin{equation}
		\label{eq:ijkl}
		X_i + X_j + X_k < 2 X_\ell,
	\end{equation}
	corresponding to all possible ways of choosing indices $i < j < k$ and $\ell$ modulo permutations of the first three. 
	Indeed determining a version amounts to choosing a four-element subset of $[m]$
	and picking one element of it to be on the right-hand side,
	which gives the $4\binom{m}{4}$ 
        in \eqref{number_of_satisfiable_inequalities}.
        Assuming $X_1 \le X_2 \le \dots \le X_m$ without loss of generality,
        a version of the inequality is automatically violated if the right-hand side index $\ell$
        is the smallest or second smallest of the four selected, or equivalently if $\ell < j$.
        This partitions the set of all versions of the
        inequality into two disjoints subsets
	\begin{align}
		\label{eq:ijkl_subsets}
		\begin{split}
		\Bigl\{(i,j,k,\ell)\in & [m]^4 :
            i<j <k, \: i \neq \ell < j
        \Bigr\} \\
	& \cup
        \Bigl\{(i,j,k,\ell)\in [m]^4 :
            i<j <k, \: j < \ell \neq k
        \Bigr\},
		\end{split}
	\end{align}
	where only the quadruples of the second subset need to be considered further to determine a maximal feasible subsystem.
        Both subsets have cardinality $2\binom{m}{4}$.

        For $m = 6$, this leaves $30$ versions of the inequality~\eqref{eq:ijkl}.
        Our next task is to understand how many of these are jointly satisfiable.
        Among these $30$ are the six versions
	\begin{equation}
		\label{eq:six_ineqs}
		\begin{aligned}
			X_3 + X_4 + X_5 < 2 X_6, &&&&&& X_3 + X_4 + X_6 < 2 X_5, \\
			X_1 + X_2 + X_5 < 2 X_3, &&&&&& X_1 + X_2 + X_6 < 2 X_3, \\
			X_1 + X_3 + X_6 < 2 X_4, \\
			X_2 + X_4 + X_6 < 2 X_5.
		\end{aligned}
	\end{equation}
	The four versions on the left are already not jointly satisfiable, since adding them results in a contradiction by $X_1, X_2 \ge 0$.
	And likewise for the two inequalities on the right, where one now also uses $X_3 \le X_4$ and $X_5 \le X_6$ in addition.
	Therefore at least two versions of the $30$ must be violated, resulting in an upper bound of $28$ for a maximal feasible subsystem.
	And indeed one can check that the assignment
	\[
		X_1 = X_2 = 0, \qquad X_3 = 4, \qquad X_4 = 6, \qquad X_5 = X_6 = 7.
	\]
	satisfies all $30$ versions under consideration except for the bottom one on each side of~\eqref{eq:six_ineqs}.
	Therefore the maximal feasible subsystems consist of $28$ inequalities, and we can improve~\eqref{eq:m4} to
	\[
		\sup_\mu \P[\mu^{\otimes 4}]{X_1 + X_2 + X_3 < 2 X_4} \le \frac{28}{60} = \frac{7}{15}.
	\]
	Since this is smaller than the upper bound in~\eqref{eq:m4}, we conclude that the upper bounds of \Cref{sup_mu_ineq} do improve with increasing $m$ in general.

\subsection{
	Upper bound for larger $m$ by Mixed Integer Linear Programming
}

While we have found a maximal feasible subsystem for $m=6$ by hand, the combinatorial nature of the problem makes
it impractical to do so for larger $m$.
Since the maximum feasible subsystem problem (with strict linear inequalities) is NP-hard in general~\cite[Theorem~4]{AK}, we do not expect there to be a simple solution to the problem for large $m$, whether by hand or algorithmically.

Still, one may hope that this specific instance is endowed with
a specific structure that makes it possible to scale the strategy
\eqref{eq:six_ineqs} algorithmically to larger $m$.
As we will see now, maximum feasible subsystem problems can be cast as Mixed Integer Linear Programs (MILP) and solved with software such as Gurobi or CPLEX.
Our particular approach generalizes and scales the considerations we used above in the $m=4$ and $m=6$ cases to up to $m=20$ using two MILPs.

The first MILP that we employ attempts to find real numbers
$0\le x_1 \le x_1 \le \dots \le x_m\le 1$ satisfying as many of the constraints
\begin{equation}
x_i + x_j + x_k \le 2 x_\ell - \operatorname{threshold}
\label{ineq_threshold}
\end{equation}
as possible, where threshold is a small positive constant
used to force a strict inequality, since MILP constraints require
non-strict inequalities for numerical stability.
First, for a given integer $m\ge 6$, we consider the second set in \eqref{eq:ijkl_subsets},
\begin{equation}
\mathcal T \coloneqq
        \Bigl\{(i,j,k,\ell)\in [m]^4 :
            i<j <k, \: j < \ell \neq k
        \Bigr\}.
\label{def_mathcal_T}
\end{equation}
The set $\mathcal T$ contains all the versions of the inequalities
that cannot be discarded right away by the argument directly following
\eqref{eq:m4}.
Based on this, the first MILP
we consider is given by
\begin{align}
    \mbox{maximize}\quad& \sum_{t\in \mathcal T} y_t,
    \label{milp1}
    \\
    \mbox{with respect to}\quad& y_t\in\{0,1\} && \text{for each }t\in\mathcal T,
    \nonumber
    \\
                                  & x_i\in [0,1] && \text{for each }i\in [m],
    \nonumber
                                  \\
    \mbox{subject to}\quad
    \nonumber
                          & x_i + x_j + x_k + 
                          M(y_t-1)
                          \le 2 x_{\ell} - \operatorname{threshold}
                          &&\text{for each } t = (i,j,k,\ell) \in \mathcal T \\
			& 0 \le x_1 \le x_2 \le \dots \le x_m,
    \nonumber
\end{align}
where $M \coloneqq 3+\operatorname{threshold}$.
The term $M(y_t-1)$ in the first constraint
is the standard ``big-M method'' to encode the implication
\begin{equation}
    \label{ineq_implication_milp}
    y_t = 1
    \quad
    \Rightarrow
    \quad
    x_i + x_j + x_k + \le 2 x_{\ell} - \operatorname{threshold}
\end{equation}
as a linear inequality involving a binary variable $y_t\in\{0,1\}$.
Indeed, if $y_t=0$ the constraint is always satisfied for
any values $x_i,x_j,x_k,x_\ell\in[0,1]$ thanks to the $M(y_t-1)$ term,
while if $y_t=1$ the $M(y_t-1)$ term vanishes and
$x_i + x_j + x_k \le 2 x_{\ell} - \operatorname{threshold}$
must hold. This shows that \eqref{ineq_implication_milp}
is equivalent to the first constraint.\footnote{Actually, our implementation
using the Gurobi solver works with the implication~\eqref{ineq_implication_milp} directly as it supports logical constraints of the form \eqref{ineq_implication_milp} without the need to explicitly introduce the big-M term.}
Our implementation also sets $x_m = 1$ without loss of generality.

We have implemented this MILP in Python using the Gurobi solver, and our implementation \texttt{milp1.py} is available on the arXiv together with this paper.
This MILP terminates on a desktop computer
for up to $m=20$ with $\operatorname{threshold}=2\cdot 10^{-5}$.
The solution is given by\footnote{These values are rounded to $7$ decimal places from the output of the Gurobi solver. We have verified that the number of satisfied inequalities is the same as without rounding, even though this is not actually needed for our proof of the upper bound.}
\begin{align}
          x_1  & = 0
        & x_2  & = 0
        & x_3  & = 0.0000020
        & x_4  & = 0.0000020
        \nonumber \\
        x_5  & = 0.0000340
        & x_6  & = 0.0000500
        & x_7  & = 0.0000580
        & x_8  & = 0.0000620
        \nonumber \\
        x_9  & = 0.0000620
        & x_{10} & = 0.5000630
        & x_{11} & = 0.7500635
        & x_{12} & = 0.9918720
        \nonumber \\
        x_{13} & = 0.9959680
        & x_{14} & = 0.9980160
        & x_{15} & = 0.9990400
        & x_{16} & = 0.9995520
        \nonumber \\
        x_{17} & = 0.9998080
        & x_{18} & = 0.9999360
        & x_{19} & = 1
        & x_{20} & = 1.
        \label{solution_milp}
\end{align}
The number of linear inequalities \eqref{ineq_threshold}
is $|\mathcal T|=9690$. Among these, $8076$ inequalities are satisfied
and the remaining $1614$ inequalities are violated.
Following the argument for $m=20$ in \eqref{eq:six_ineqs},
this suggests an upper bound of
\begin{equation}
	\label{m16_upper_bound}
	\frac{8076}{19380}\approx 0.41672
\end{equation}
for the problem
\eqref{4_eq_1}, where the denominator is
\eqref{number_of_satisfiable_inequalities} for $m=20$.

\begin{remark}
    \label{rem:redundant}
	Resolution of the MILP \eqref{milp1} can be sped up in practice by adding redundant constraints.
	One family of such redundant constraint is
	\begin{align*}
		y_{(i+1,j,k,\ell)} & \le y_{(i,j,k,\ell)}, &
		y_{(i,j+1,k,\ell)} & \le y_{(i,j,k,\ell)}, \\
		y_{(i,j,k+1,\ell)} & \le y_{(i,j,k-1,\ell)}, &
		y_{(i,j,k,\ell-1)} & \le y_{(i,j,k,\ell)},
	\end{align*}
	for all tuples of indices $(i,j,k,\ell) \in \mathcal T$ for which this makes sense.
	For instance, the leftmost constraint holds because $X_{i+1}+X_j+X_k < 2X_{\ell}$ implies $X_i + X_j + X_k < 2 X_{\ell}$ due to $X_i \le X_{i+1}$.
	Another family of redundant constraints that speeds up the MILP resolution consists of pre-computed subsystems of 2 or 3 inequalities that are unsatisfiable.
	For instance, if $\mathcal A\subseteq\mathcal T^2$ is a precomputed subset of pairs of versions of the inequality that are not jointly satisfiable, and $\mathcal B\subseteq\mathcal T^3$ a precomputed subset of triples of versions of the inequality that are not jointly satisfiable, we may add the redundant constraints
	$$
	y_s + y_t \le 1 ~\text{ for each }~ (s,t)\in \mathcal A, \qquad y_u + y_v + y_w \le 2 ~\text{ for each }~ (u,v,w) \in\mathcal B.
	$$
	These redundant constraints let the Gurobi solver terminate for $m=20$ in a few seconds.
	Even for $m=25$, it terminates overnight and produces the tighter upper bound given by $\frac{20898}{50600} \approx 0.413004$.
\end{remark}

If the MILP terminates, its optimal value tells us the maximal number of jointly satisfiable inequality versions
$x_i+x_j+x_k\le 2 x_\ell - \text{threshold}$.
However, due to the threshold, this does not necessarily coincide with the maximal number of jointly satisfiable strict inequalities we are interested in.
We therefore do not yet get a formal proof of the upper bound in \eqref{4_eq_1} from the MILP solution.

In order to formally prove the upper bound in \eqref{4_eq_1},
we mimic the strategy explained in \eqref{eq:six_ineqs} as follows.
From the solution to \eqref{milp1},
we obtain a formally provable upper bound using a second MILP, whose goal
is to output disjoint subsystems of inequalities
as in \eqref{eq:six_ineqs}, with each subsystem unsatisfiable.
The solution to the MILP \eqref{milp1} above provides 
$\mathcal T = \mathcal T^s \cup \mathcal T^v$, where
$\mathcal T^s$ contains the tuples $t=(i,j,k,\ell)$ corresponding to inequalities \eqref{ineq_threshold}
satisfied by the solution, while
$\mathcal T^v$ contains those that are
violated. The second MILP we consider is the following:
\begin{align}
    \mbox{find any feasible}\quad& 
    y_{tu}\in\{0,1\} &&\text{for each }(t, u)\in\mathcal T^s \times \mathcal T^v,
    \nonumber
    \\[2pt]
    \mbox{subject to}\quad&
    \sum_{u \in \mathcal T^v} y_{tu} \le 1 &&\text{for each } t \in \mathcal T^s,
    \label{milp2}
    \\
    \nonumber
   &F(q, u)
    +
    \sum_{t\in\mathcal T^s}
    F(q, t)
    y_{tu}
    \ge 0
   &&\text{for each $q\in[m]$ and $u \in\mathcal T^v$},
\end{align}
where
$F(q, t)$ is defined
for general $t=(i,j,k,\ell)\in\mathcal T$
by
\[
	F(q, t) \coloneqq
    \delta_{q\le i}
    + \delta_{q\le j}
    + \delta_{q\le k}
    - 2 \delta_{q\le \ell}.
\]
The reason for considering this MILP is as follows.

\begin{lemma}
	\label{lem:subsystems}
	If the MILP \eqref{milp2} is feasible, then $\mathcal T^s$ is the set of indices of a maximal feasible subsystem of the inequalities $X_i + X_j + X_k < 2 X_\ell$.
\end{lemma}

\begin{proof}
We explain how a feasible solution amounts to a certificate of maximality for $\mathcal T^s$.
As mentioned, the idea is analogous to the argument for $m=6$ in \eqref{eq:six_ineqs}.

Given a feasible solution $(y_{tu})_{t\in\mathcal T^s, u\in\mathcal T^v}$, for each $u \in T^v$ we consider the inequality associated to $u$
together with all those associated to the $t \in \mathcal T^s$ with $y_{tu}=1$.
Then the first constraint $\sum_{u\in\mathcal T^v} y_{tu} \le 1$ for every $t$ is precisely what enforces these systems of inequalities to be disjoint as $u$ varies.

The second constraint involving the function $F$
ensures that for each $u$, the system of inequalities described in the previous paragraph is infeasible.
The idea is that as in~\eqref{eq:six_ineqs}, one simply sums these inequalities and uses $0 \le x_1\le x_2\le \dots \le x_m$ to arrive at a contradiction.
To see how this works, fix some $u\in\mathcal T^v$ and sum the inequalities 
\begin{equation}
    x_i+x_j+x_k - 2x_l < 0  \quad\text{ for }\quad (i,j,k,l) 
\in
\{u\} \cup \{t\in\mathcal T^s:y_{tu}=1\}
\label{inequalities_to_sum}
\end{equation}
each with equal weight 1. If all inequalities in \eqref{inequalities_to_sum}
hold, this implies that 
\begin{equation}
	\label{eq:contradiction_ineq}
	a_1 x_1 + a_2 x_2 + \dots + a_m x_m < 0
\end{equation}
for $a_i = F(i, u) - F(i-1,u) + \sum_{t\in\mathcal T^s: y_{tu}=1} F(i, t) - F(i-1, t)$
for each $i=1,\dots,m$. To find a contradiction,
since all we know about $(x_i)_{i=1}^m$ is that they are
non-negative and non-decreasing, 
we rewrite the left-hand
side above as a linear combination of the $(x_q-x_{q-1})_{q=1}^m$, 
where $x_0 \coloneqq 0$ for convenience. Denoting by
\[
S_q \coloneqq F(q, u) + \sum_{t \in\mathcal T^s : \: y_{tu}=1} F(q, t) = a_q + \dots + a_m 
\]
the left-hand side of the constraint in the MILP \eqref{milp2},
this takes the form
\[
	\sum_{q=1}^m a_q x_q
	= \sum_{q=1}^m S_q x_q - \sum_{q=1}^{m-1} S_{q+1} x_q
	= \sum_{q=1}^m S_q (x_q - x_{q-1}).
\]
Since $S_q \ge 0$ thanks to the MILP constraint,
and $x_q - x_{q-1} \ge 0$ for all $q \in [m]$, we obtain that the left-hand side of~\eqref{eq:contradiction_ineq} is non-negative, contradicting the strict inequality.
Therefore for every monotone assignment of numbers to the variables $x_1, \dots, x_m$, at least one of the strict inequalities in \eqref{inequalities_to_sum} we summed up must be violated.

To sum up, we have $|\mathcal T^v|$ disjoint systems of inequalities, each of which is infeasible.
Therefore any feasible subsystem of the original inequalities must violate at least one inequality in each of these systems.
Thus the maximal number of inequalities that can be satisfied is $|\mathcal{T} \setminus \mathcal{T}^v| = |\mathcal T^s|$, as was to be shown.
\end{proof}

There is no reason, a priori, for the MILP \eqref{milp2} to be feasible
because summing the inequalities with uniform weights as we do above may not be a necessary
condition for the corresponding subsystem to be infeasible.\footnote{By Farkas' lemma, summing with general nonnegative coefficients gives a necessary and sufficient condition.}
In other words, one would not expect the converse of \Cref{lem:subsystems} to hold in general, unless there is additional structure in the systems under consideration that would allow a reduction to the case of uniform weights.

Again we have implemented this MILP in Python using the Gurobi solver, and our implementation \texttt{milp2.py} is available on the arXiv together with this paper.
Taking $m = 20$ and using the
set of inequalities $\mathcal{T}^s \cup \mathcal{T}^v$ returned by the first MILP, which we make available as \texttt{ineqs\_m\_is\_20.log},
the MILP \eqref{milp2} terminates and outputs
a feasible solution $(y_{tu})_{t\in\mathcal T^v, u\in\mathcal T^s}$.
Hence \Cref{lem:subsystems} applies and provides a certificate of maximality for $\mathcal T^s$, consisting of $|\mathcal T^v| = 1614$ disjoint subsystems of inequalities, each infeasible.
For a formal proof of maximality, one can now also check the infeasibility of each subsystem by summing the corresponding inequalities, and verifying in addition that these systems are disjoint.
Since $|\mathcal T^s| = 8076$ and
$\frac{8076}{19380} = \frac{673}{1615}$,
this finishes the proof of the upper bound in~\eqref{4_eq_1}.\footnote{Similarly for $m = 6$, our implementation recovers the two systems displayed in \eqref{eq:six_ineqs}, which we had used to prove the upper bound of $\frac{28}{60}$.}
We offer the disjoint systems of inequalities produced by the second MILP as \texttt{optimality\_witness\_m\_is\_20.log}, and a script \texttt{verify\_witness.py} to check their disjointness and infeasibility under the monotonicity constraint.

Although the first MILP has terminated for up to $m = 25$, the second MILP has only terminated up to $m = 20$ so far, which is why we do not yet have a rigorous proof of the improved upper bound mentioned in \cref{rem:redundant}.

\section{Open problems}
\label{section_open_problems}

\subsection{The lower bound conjecture}

Perhaps the most interesting open question is to determine the exact value of~\eqref{4_eq_1}.

\begin{conj}
	\label{conj:lower_bound_tight}
	The lower bound of \cref{prop:lower_bound_04} is the exact value:
	\[
		\sup_\mu \P[\mu^{\otimes 4}]{X_1 + X_2 + X_3 < 2 X_4} = \sup_{p \in (0,1)} \frac{p(2-p)}{1 + p + p^2 + p^3}.
	\]
\end{conj}

The following piece of evidence makes this plausible.
In the proof of \Cref{prop:lower_bound_04}, the distributions achieving the lower bound
were obtained
by combining $\delta_0$ with a uniform distribution
over the points $(1-2^{-i})_{i=1}^N$ and repeating this pattern
at a smaller scale using the construction of \Cref{strict_ineq_geom}.
The striking observation is now that this is mirrored in the output of the first MILP \eqref{milp1}.
Despite having no knowledge of our strategy involving the pattern of the distributions which led to the lower bound,
the solutions to the first MILP \eqref{milp1} often mimic this pattern: for example at $m = 15$, up to rescaling by $1.0030$ and small numerical discrepancies, we have at the top scale
\[
	x_8 = 1 - 2^{-1}, \qquad x_9 = 1 - 2^{-2}, \qquad \dots, \qquad x_{14} = 1 - 2^{-7},
\]
and at the next scale,
\[
	x_4 = (1 - 2^{-1}) \cdot \eta, \qquad x_5 = (1 - 2^{-2}) \cdot \eta, \qquad x_6 = (1 - 2^{-3}) \cdot \eta,
\]
with $\eta = 0.00016$ as in~\Cref{prop:lower_bound_04}.
On the other hand, $x_7 = x_6$ breaks the pattern---this might
be due to boundary effects as $x_7$ is the last point
before the upper scale.

The same phenomenon can be seen in the solutions of the first MILP \eqref{milp1}
for other values of $m$, as long as $m$ is not too small to reach the second scale.
For example, it is also present in the $m = 20$ solution of \eqref{solution_milp}, although it is less pronounced there.
For $m=25$, termination of the MILP \eqref{milp1} with the redundant
constraints of \Cref{rem:redundant} outputs the optimal solution
(rounded for presentation):
\begin{align*}
    \text{\underline{Scale 4}: }~~ x_1 &= 0, & x_2 &= 0.
                        \\
    \text{\underline{Scale 3}: }~~x_3 &= 4 \cdot 10^{-7}, & x_4 &= 6 \cdot 10^{-7} & x_5 &= 6 \cdot 10^{-7}.
                        \\
    \text{\underline{Scale 2}: }~~x_6 & = 6.1 \cdot 10^{-5}, & x_7 & = 9.2 \cdot 10^{-5}, & x_8 & = 10.7 \cdot 10^{-5}, \\
			x_9 & = 11.5 \cdot 10^{-5} & x_{10} & = 12.2 \cdot 10^{-5}, & x_{11} & = 12.2 \cdot 10^{-5}.
                        \\
                        \text{\underline{Scale 1}: }
                        x_{12} & = 0.5001, & x_{13} & = 0.7502, & x_{14} & = 0.8751, \quad x_{15}  = 0.9376, \\
			x_{16} & = 0.9688, & x_{17} & = 0.9845, & x_{18} & = 0.9923, \quad x_{19}  = 0.9962, \\
			x_{20} & = 0.9981, & x_{21} & = 0.9991, & x_{22} & = 0.9996, \quad x_{23}  = 0.9998, \\
			x_{24} & = x_{25} = 1.
\end{align*}
Again, the MILP solution divides the variables into different scales,
and within one scale the variables roughly follow the exponential spacing
pattern $\eta \cdot (1 - 2^{-i})$.
While not formal evidence, the fact that the MILP solutions mimic
both the multiple-scales and the exponential spacing pattern
of the lower bound  \eqref{eq:lower_bound}
supports \cref{conj:lower_bound_tight}.

By using faster computers than the desktop used to solve \eqref{milp1},
it is likely that the upper bound can be pushed down
a little further by solving the MILPs \eqref{milp1} and \eqref{milp2}.
However, the numbers of variables and constraints grow quickly with $m$,
since the set $\mathcal T$ in \eqref{def_mathcal_T}
has cardinality of order $m^4$. The techniques in \Cref{rem:redundant}
help to terminate the first MILP \eqref{milp1} up to $m = 25$,
suggesting an upper bound of $4.13004$.
For $m=25$ though, the second MILP \eqref{milp2} is the bottleneck as it does not terminate in less than a day for $m>20$.
Smarter MILPs than the ones used above may also push the upper
bound further down, perhaps by leveraging more of the symmetries of the
problem.
However, by nature these numerical methods are unable to determine the
theoretical limit of this strategy as $m\to+\infty$ and
see if it converges to the lower bound \eqref{eq:lower_bound}.

We are not aware of existing mathematical ideas that could solve \cref{conj:lower_bound_tight},
either building on the exchangeability
strategy of the previous sections or by leveraging completely different
arguments.
We hope that the present work, and the apparent difficulty 
of determining $\sup_{\mu}\P[\mu^{\otimes 4}]{X_1+X_2+X_3<2X_4}$,
will raise interest in problems of this type
and lead to new arguments resolving \cref{conj:lower_bound_tight}.

\subsection{Attainability of the supremum}

Another interesting open question is the following.

\begin{prob}
	\label{prob:open}
	For coefficients $c_1, \dots, c_n \in \R$ with $\sum_{i=1}^n c_i < 0$, when is the supremum $\sup_{\mu} \P[\mu^{\otimes n}]{\sum_{i=1}^n c_i X_i > 0}$ achieved, where $\mu$ ranges over probability measures on $\R_+$?
\end{prob}

As a first observation, if $\mu$ has an atom at $0$, then \Cref{strict_ineq_geom} shows that it is not optimal.
Therefore any putative optimizer must satisfy $\mu(\{0\}) = 0$.
Furthermore an optimal $\mu$ cannot be finitely supported, since then (as already noted in \Cref{example}) one can assume $\mu$ to have an atom at $0$, which we have just ruled out.


%

\subsection{Fair games}

Instead of focusing on fixed coefficients $c_1,\dots,c_n$, one can ask
for specific coefficients that make games like Beat the Average fair.
\begin{prob}
    Determine a real $c_*\in(2,3)$ such that
    \begin{equation}
        \sup_{\mu} \P[\mu^{\otimes 4}]{X_2 + X_2 + X_3 < c_* X_4} = \frac 12,
        \label{sup_fiair_game_c_*}
    \end{equation}
    where $\mu$ ranges over all probability measures on $\R_+$, or show that no such $c_*$ exists.
\end{prob}

It has been observed by Vincent Yu\footnote{See \href{https://mathoverflow.net/a/484853/27013}{mathoverflow.net/a/484853/27013}.} that the left-hand side of~\eqref{sup_fiair_game_c_*} is not continuous in $c_*$, so that the intermediate value theorem does not apply.
	To see that continuity fails for example at $c_* = 2$, take the Bernoulli measure $\mu = (1 - p) \delta_0 + p \delta_1$ for any $c_* > 2$.
	This makes the inequality fail only if $X_4 = 0$ or $X_1 = X_2 = X_3 = X_4 = 1$, so that
	\begin{align*}
            \P[\mu^{\otimes 4}]{X_1 + X_2 + X_3 < c_* X_4} &= 1 - (1 - p) - p^4,
            \\ \sup_\mu \P[\mu^{\otimes 4}]{X_1 + X_2 + X_3 < c_* X_4} &\ge \sup_{p \in [0,1]} \left( p - p^4 \right) = \tfrac{3}{4^{4/3}} > 0.472.
	\end{align*}
	Since this is larger than our upper bound~\eqref{4_eq_1} at $c_* = 2$, we conclude that the continuity fails at $c_* = 2$.

\subsection{Obtuse random triangles}

Consider three independent and identically distributed random points in $\R^2$ with distribution $\mu$.
These three points form a triangle (possibly degenerate).
What is the probability for this triangle to be obtuse, i.e., to have an angle larger than $\pi/2$?
This question goes back to Lewis Carroll's \emph{Pillow Problems} from 1893~\cite{portnoy}.

Of course, the answer will depend on the distribution $\mu$, so a more well-defined question is: how large and how small can one make the probability of an obtuse triangle by varying $\mu$? This was studied very recently by Drivas and Retakh~\cite{DR}.
They first observe that by taking $\mu$ to be supported on an arc smaller than a semicircle, one can achieve probability $1$ for an obtuse triangle.
In the other direction, they prove that the minimal probability lies in the interval $[\frac{1}{3}, \frac{4}{9}]$, where the exact value remains open.

Given that the condition for a triangle to be obtuse also takes the form of a strict inequality, it is natural to ask whether the techniques of the present paper can be adapted to this problem.\footnote{We thank Daniel Lacker for finding~\cite{DR} and suggesting this question to us.}
Note that the precise methods we used in~\Cref{example} do not apply directly, since the inequality characterizing obtuseness is not linear in the coordinates of the points.

\printbibliography

\end{document}